\documentclass[11pt]{amsart}
\usepackage[utf8]{inputenc}

\usepackage{overpic, amssymb, times, graphicx, tikz-cd, yfonts, thmtools, cancel, wrapfig}
\usepackage[all]{xy}
\usepackage[backref=page]{hyperref}
\usepackage{verbatim}
\usepackage{enumitem}
\usepackage{mathabx}

\hypersetup{
    colorlinks=true,
    linkcolor=blue,
    filecolor=blue,
    urlcolor=blue,
    citecolor=blue,
}
\DeclareMathOperator{\pow}{pow}

\DeclareMathOperator{\codim}{codim}

\DeclareMathOperator{\totalspace}{Tot} 

\DeclareMathOperator{\Crit}{Crit}
\DeclareMathOperator{\Critv}{Critv}
\DeclareMathOperator{\im}{im} 
\DeclareMathOperator{\ind}{ind} 
\DeclareMathOperator{\Aut}{Aut}

\DeclareMathOperator{\Span}{span}

\DeclareMathOperator{\Symp}{Symp}

\DeclareMathOperator{\Id}{Id}

\DeclareMathOperator{\Flow}{Flow}

\DeclareMathOperator{\PD}{PD}

\DeclareMathOperator{\Line}{\mathcal{L}}

\newcommand{\thicc}[1]{\pmb{#1}}

\newcommand{\Proj}{\mathbb{CP}}
\newcommand{\ProjOne}{\Proj^{1}}

\newcommand{\C}{\mathbb{C}}
\newcommand{\R}{\mathbb{R}}
\newcommand{\Z}{\mathbb{Z}}

\newcommand{\disk}{\mathbb{D}}

\newcommand{\bigO}{\mathcal{O}}

\newcommand{\delbar}{\overline{\partial}}

\newcommand{\Cinfty}{\mathcal{C}^{\infty}}

\newcommand{\Sthree}{(S^{3},\xi_{std})}

\newcommand{\Circle}{\mathbb{S}^{1}}

\newcommand{\half}{\frac{1}{2}}

\newcommand{\be}{\begin{enumerate}}
\newcommand{\ee}{\end{enumerate}}

\newcommand{\sphere}{\mathbb{S}}

\newcommand{\Mxi}{(M,\xi)}

\newcommand{\norm}[1]{\left\lVert#1\right\rVert}

\newcommand{\multisec}{\thicc{\mathfrak{s}}}

\newcommand{\ModSpace}{\mathcal{M}}

\newcommand{\hypersurface}{W}

\newtheorem{thm}{Theorem}[subsection]

\newtheorem{ex}[thm]{Example}

\newtheorem{defn}[thm]{Definition}

\newtheorem{lemma}[thm]{Lemma}

\newtheorem{rmk}[thm]{Remark}

\newtheorem{edits}[thm]{Editor notes}

\usepackage[refpage,notocbasic]{nomencl}
\makenomenclature

\title{Symplectic mapping class relations from pencil pairs}
\author{Russell Avdek}
\date{\today}

\begin{document}

\begin{abstract}
We describe symplectic mapping class relations between products of positive Dehn twists along Lagrangian spheres in Weinstein $4$-manifolds, all of which are affine $\C$ varieties. The relations are obtained by applying classification results for Fano $3$-folds and polarized $K3$ surfaces of small genus to a general methodology -- finding pencil pairs.
\end{abstract}
\maketitle

\numberwithin{equation}{subsection}
\setcounter{tocdepth}{1}

\section{Introduction}

This article describes a general method for finding symplectic mapping class relations between products of positive Dehn twists on (affine algebraic) Weinstein manifolds $F$ of $\dim = 2n$ by studying complex projective varieties of $\dim_{\C}=n+1$. The general method consists of finding \emph{pencil pairs} and is known to experts. It recovers known mapping class relations on punctured Riemann surfaces such as the lantern relation \cite{AS:Pencil, Dehn:CollectedPapers, Johnson:Lantern}, described in \S \ref{Sec:GeneralizedLantern} and depicted in Figure \ref{Fig:Lantern}. In formally defining pencil pairs and working out some basic techniques for finding them (see \S \ref{Sec:ProjectiveModel}), we hope to inspire specialists in symplectic topology and algebraic geometry to further study the case $\dim F \geq 4$, for which few examples of such symplectic mapping class relations are known. Therefore our exposition will cover some relevant basics from both subject areas.

Our search for mapping class relations on high-dimensional Weinstein manifolds is motivated by recent advances in high-dimensional contact topology \cite{Breen:Folded, BHH, HH:Convex}. In \S \ref{Sec:ContactMotivations} we describe how these advances imply that relations between products of positive Dehn twists on Weinstein manifolds of any even dimension exist in abundance.

\begin{figure}[h]
\begin{overpic}[scale=.5]{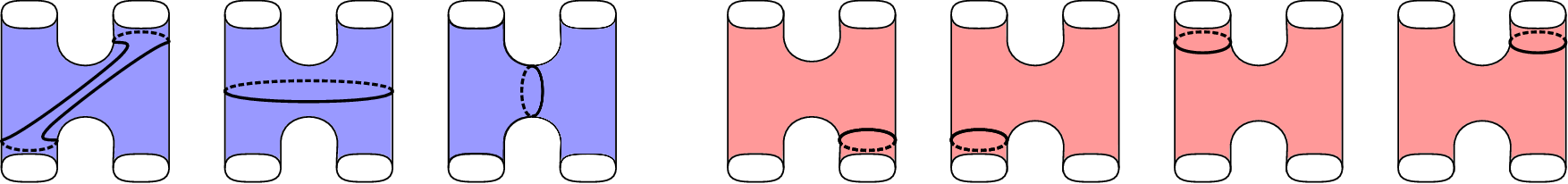}
\put(33, -1){$a$}
\put(19, -1){$b$}
\put(5, -1){$c$}
\end{overpic}
\vspace{2mm}
\caption{The lantern relation states that the product $\tau_{c}\tau_{b}\tau_{a}$ of Dehn twists along circles on the left is boundary-relative isotopic to a product of Dehn twists along the circles on the right. Throughout, products represent compositions of maps, so $\tau_{c}\tau_{b}\tau_{a}$ means that we apply $\tau_{a}$, then $\tau_{b}$, and finally $\tau_{c}$. The surfaces depicted are oriented by a normal vector which sticks out of the page along their front halves.}
\label{Fig:Lantern}
\end{figure}

In \cite{Oba:FourDRelation} Oba finds a symplectic mapping class relation on a Weinstein $4$-manifold by applying the classification of del Pezzo surfaces \cite{Fujita} to find a pencil pair. Our main application leverages classification results for polarized $K3$ varieties \cite{Mukai:Gleg10, SD:ProjectiveModel} and Fano $3$-folds \cite{Iskovskih:I, Iskovskih:II, MoriMukai} to find new pencil pairs, and so new examples of mapping class relations on Weinstein $4$-manifolds. Here is one such family of examples.

\begin{thm}\label{Thm:MainExample}
Let $Z \subset \ProjOne \times \ProjOne \times \Proj^{2}$ be the $K3$ surface cut out by a smooth complete intersection of $\deg=(1,1,2)$ and $\deg=(1,1,1)$ divisors. Let $(F, \beta_{F})$ be the Weinstein manifold given by the complement of a $\deg=(1,1,1)$ divisor in $Z$, with the symplectic structure determined by a $\deg=(1,1,1)$ line bundle. Then there is a compactly supported symplectomorphism $\tau_{\partial F} \in \pi_{0}\Symp^{c}(F, d\beta_{F})$ which can be factored into a product of $m$ Dehn twists along Lagrangian spheres in $(F, \beta_{F})$, where
\begin{equation*}
m=56, 60, 62, 64, 78.
\end{equation*}
\end{thm}

We found six such families of mapping class relations in total by applying the Fano $3$-fold and $K3$ surface classifications. One recovers the relation of \cite{Oba:FourDRelation} (see \S \ref{Sec:ObaRevisited}). Because the number of Dehn twists in our relations are all distinct, they cannot be derived from well-known braid relations \cite{KhovanovSeidel}. See \S \ref{Sec:Braids}. The $\tau_{\partial F}$ above is a fibered Dehn twist \cite{CDvK} and all mapping class relations coming from pencil pairs are determined by factorizations of such $\tau_{\partial F}$. See \S \ref{Sec:AlgebraicPencilReview}.

Every pencil pair gives rise to an infinite family of pencil pairs by a generally-applicable ``cabling'' result described in Theorem \ref{Theorem:Cabling}. Consequently, each collection of mapping class relations we have described above determines infinitely many mapping class relations. The proof of Theorem \ref{Theorem:Cabling} relies on a general result concerning the symplectic topology of high-degree hypersurfaces in projective varieties, Theorem \ref{Theorem:BranchedCover}, which may be of independent interest.

\subsection{Conventions}

All algebraic varieties in this article are smooth and defined over $\C$. Throughout $\dim$ and $\dim_{\C}$ refer to real and complex dimension, respectively. On a closed manifold $X$ we use $\PD_{X}$ for Poincar\'{e} duality in either direction $H_{k}(X) \leftrightarrow H^{\dim X -k}(X)$. Cotangent disk bundles (determined by some Riemannian metric) over a smooth manifold $X$ will be denoted $\disk^{\ast}X \subset T^{\ast}X$. On $T^{\ast}X$, the canonical $1$-form is denoted $pdq$ whose differential, denoted $dp \wedge dq$ is the canonical symplectic form. All other symplectic manifolds in this article will be complex projective varieties $X, Z, B, \ModSpace \subset \Proj^{n}$ or affine varieties $F,W \subset \C^{n}$, with their symplectic structures inherited from the ambient space. For a function $h$ on a symplectic manifold with symplectic form $\omega$, the Hamiltonian vector field $X_{h}$ is defined by $dh = \omega(\ast, X_{h})$.

\subsection{Acknowledgments}

\begin{wrapfigure}{r}{0.07\textwidth}
	\centering
	\includegraphics[width=.07\textwidth]{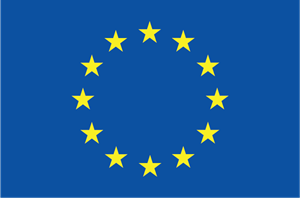}
\end{wrapfigure}

We thank Eduardo Alves da Silva, Yankı Lekili, and Takahiro Oba for interesting discussions. We're especially grateful to Marco Golla for discussions around complex surfaces, reviewing an early draft of the article, and for helping to expand the list of examples in \ref{Sec:GeneralizedLantern}. We also thank our anonymous referees for their thoughtful feedback which has greatly improved the quality of this article. This project has received funding from the European Union’s Horizon 2020 research and innovation programme under the Marie Skłodowska-Curie grant agreement No 101034255. We thank Cofund MathInGreaterParis and the Fondation Mathématique Jacques Hadamard for supporting us as a member of the Laboratoire de Math\'{e}matiques d'Orsay at Universit\'{e} Paris-Saclay during the completion of the first draft of the article. Finally, we thank the CNRS for supporting us as a member of the Institut de Mathématiques de Jussieu-Paris Rive Gauche at Sorbonne Université.

\section{Background}

Here we cover background material and provide motivation for our interest in finding symplectic mapping class relations between products of positive Dehn twists.

\subsection{Symplectic mapping class relations}

Let $(F, \beta_{F})$ be an \emph{exact symplectic manifold}, meaning that $\beta_{F} \in \Omega^{1}(F)$ is such that $\omega_{F} = d\beta_{F}$ is symplectic. We require that $F$ is a compact Liouville domain -- meaning that the vector field $X_{F}$ determined by $d\beta_{F}(X_{F}, \ast ) = \beta_{F}$ points transversely out of $\partial F$ -- or the completion of a compact Liouville domain, meaning in any case that it has finite type. We say that $(F, \beta_{F})$ is \emph{Weinstein} if $X_{F}$ is gradient-like for some function $f$ on $F$. See \cite{SteinToWeinstein} for details and background. The canonical examples of Weinstein manifolds are affine algebraic varieties $F \subset \C^{N}$ with $f = \norm{z}^{2}$ and $\beta = -df \circ J_{0}$, where $J_{0}$ is the standard complex structure on $\C^{N}$. The group of compactly supported symplectomorphisms will be denoted $\Symp^{c}(F, \omega_{F})$ and we are interested in studying $\pi_{0}\Symp^{c}(F, \omega_{F})$.

An \emph{exact Lagrangian sphere} $L \subset F$, where $\dim F=2n$, is a sphere $\sphere^{n}$ for which $\beta_{F}|_{TL} = df$ for some $f \in \Cinfty(L)$. A \emph{framing} on such a sphere is an identification of the sphere with the boundary of a disk, up to homotopy. We give a quick review of Dehn twists about such $L$, which appear in nature as monodromies of Milnor fibers of the singularities $\sum_{j=0}^{n} z_{j}^{2}$ on $\C^{n+1}$ \cite{Arnold:Singularity}. According to the Weinstein neighborhood theorem, every exact Lagrangian sphere $L$ has a neighborhood of the form
\begin{equation*}
(N_{F}(L), \beta_{F}) = (\disk^{\ast}\sphere^{n}, pdq + df), \quad f \in \Cinfty(\disk^{\ast}\sphere^{n}).
\end{equation*}
Assume the disk bundle is defined using the round metric of radius $1$. We define a Dehn twist $\tau_{L}$ locally as
\begin{equation*}
\tau_{L}=(-\Id_{\sphere^{n}})^{\ast}\Flow^{1}_{X_{b}}
\end{equation*}
where $b=b(\norm{p})$ is a function which equals $0$ near $\norm{p}=0$ and $\pi$ near $\norm{p}=1$, $X_{b}$ is the Hamiltonian vector field, and $(-\Id_{\sphere^{n}})^{\ast}$ is the symplectomorphism induced by the antipodal map $q\mapsto-q$ on $\sphere^{n}$. In words, $\Flow^{1}_{X_{b}}$ applies the time $b(\norm{p})$ geodesic flow to (co)vectors of length $\norm{p}$. In particular, when $\norm{p}=1$, the time $1$-flow will send a vector over $q\in \sphere^{n}$ to a vector over $-q$ and the $(-\Id_{\sphere^{n}})^{\ast}$ correction ensures that $\tau_{L}$ is the identity near $\partial \disk^{\ast}\sphere^{n}$. Extend $\tau_{L}$ to $F \setminus N_{F}(L)$ as the identity. The dependence on $b$ is irrelevant when we view $\tau_{L} \in \pi_{0}\Symp^{c}$. It can be seen by drawing pictures that when $\dim F = 2$, $\tau_{L}$ coincides with the usual notion of a (right-handed) Dehn twist along a simple closed curve on a Riemann surface. The action of $\tau_{L}$ on $H_{\ast}(F)$ is given by the \emph{Picard-Lefschetz formula}
\begin{equation}\label{Eq:PicardLefschetz}
(\tau_{L})h = \begin{cases}
h + (-1)^{\half (n+1)(n+2)}(h \bullet [L])[L] & h \in H_{n}(F),\\
h & \text{otherwise.}\\
\end{cases}
\end{equation}
Here $\bullet$ denotes the intersection pairing on $H_{n}(F)$. As a sanity check, we get
\begin{equation}\label{Eq:PLZeroSection}
[L]\bullet [L] = -\chi(L), \quad (\tau_{L})[L] = (-1)^{n+1}[L]
\end{equation}
which is expected since $\tau_{L}$ acts as the antipodal map on $L \simeq \sphere^{n}$. In the cases $n=2,6$, $\tau_{L}^{2}$ is smoothly, but not symplectically, isotopic to the identity \cite{Avdek:Liouville, Seidel:KnottedSpheres}. See \cite{KRW} for more on powers of Dehn twists.

A \emph{positive factorization} of $\phi \in \pi_{0}\Symp^{c}(F, \omega_{F})$, denoted as $\vec{L} = (L_{1}, \cdots L_{m})$, is a tuple of framed, exact Lagrangian spheres for which
\begin{equation*}
\phi = \tau_{L_{m}}\cdots \tau_{L_{1}} \in \pi_{0}\Symp^{c}(F, \omega_{F}).
\end{equation*}
We say that $m$ is the \emph{length} of the positive factorization, and write
\begin{equation*}
|\vec{L}| = m.
\end{equation*}
Two positive factorizations $\vec{L}$ and $\vec{L}' = (L_{1}', \cdots L_{m'}')$ of a $\phi$ are \emph{deformation equivalent} if $m=m'$ and for each $i$, $L_{i}$ is Hamiltonian isotopic to $L_{i}'$, whence we write $\vec{L}=\vec{L}'$. We are interested in such classes of $\vec{L}$ as they provide instructions for building Weinstein manifolds of $\dim = \dim F + 2$ as described in \S \ref{Sec:ContactMotivations}.

\begin{defn}\label{Def:PositiveRelation}
A \emph{positive relation on $F$} is a pair of positive factorizations $\vec{L}, \vec{L}'$ of some $\phi$, with $\vec{L}\neq \vec{L}'$.
\end{defn}

\subsection{Known relations and counting twists}\label{Sec:Braids}

There are a few well-known ways to modify a positive factorization, and so determine positive relations.
\be
\item \emph{The disjoint union relation}: If $L \cap L' = \emptyset$, then $\tau_{L}\tau_{L'} = \tau_{L'}\tau_{L}$.
\item \emph{The conjugacy relation}: For any two $L$ and $L'$, $\tau_{L}\tau_{L'}= (\tau_{\tau_{L}L'})\tau_{L}$.
\item \emph{The braid relation}: If $L$ and $L'$ intersect transversely in a single point, $\tau_{L}\tau_{L'}\tau_{L} = \tau_{L'}\tau_{L}\tau_{L'}$.
\ee
The first two are clear. The last can be verified directly by studying Lagrangian spheres in $A_{2}$ Milnor fibers and gives rise to an embedding of the braid group on $m$ strands into the $\pi_{0}\Symp^{c}$ of an $A_{m}$ Milnor fiber \cite{KhovanovSeidel}. When $\dim F=4$, this embedding is an isomorphism \cite{Evans:MappingClass, Wu:An}.

We seek to single out the positive factorizations which are not derived from braid groups and so are ``not obvious''. Let $T=T(F, \beta_{F})$ be the free group on Hamiltonian isotopy classes of exact framed Lagrangian spheres. Consider the normal subgroup $B=B(F, \beta_{F}) \subset T$ generated by
\be
\item $[L][L'][L]^{-1}[L']^{-1}$ when the classes have representatives $L, L'$ which are disjoint,
\item $[L][L'][L]^{-1}[\tau_{L}L']^{-1}$ for any $L, L'$,
\item and $[L][L'][L][L']^{-1}[L]^{-1}[L']^{-1}$ when the classes have representatives $L, L'$ intersecting transversely in $F$ a single point.
\ee
We have canonical homomorphisms
\begin{equation}\label{Eq:TauP}
\begin{gathered}
\tau: T \rightarrow \pi_{0}\Symp(F, \omega_{F}), \quad \tau[L] = \tau_{L},\\
p: T \rightarrow \Z, \quad p [L] = 1
\end{gathered}
\end{equation}
which factor through $T/B$. In words, $\tau$ sends formal products of exact Lagrangian spheres (and their inverses) to products of their positive (and negative) Dehn twists while $p$ measures the sum of the exponents. Therefore for a positive factorization $\vec{L}$, $p [\vec{L}] = |\vec{L}|$. By construction, $\ker \tau$ gives all relations in the subgroup $\im \tau$ of compactly supported symplectomorphisms generated by Dehn twists.

Each positive factorization $\vec{L}$ of some compactly supported symplectomorphism $\phi$ yields
\begin{equation*}
[\vec{L}]=[L_{k}]\cdots [L_{1}]
\end{equation*}
which can be viewed as an element of either $T$ or $T/B$. A positive relation $\vec{L}, \vec{L}'$ is \emph{non-braid-like} if $[\vec{L}] \neq [\vec{L}'] \in T/B$. It's clear that if $|\vec{L}|\neq |\vec{L}'|$, then the relation is non-braid-like. Examples of non-braid-like relations include the lantern and chain relations on punctured Riemann surfaces, cf. \cite[\S 15]{OS:SurgeryBook}.

We are interested in finding non-braid-like positive relations when $\dim F > 2$. In a way which is made precise in \cite{Keating:Stabilization}, the existence of such relations is more restrictive than the $\dim F = 2$ case. The following lemma gives an easy way to see that relations in the $\dim F=2$ case (such as the lantern, which has $p=\pm 1$) cannot be straightforwardly generalized to higher dimensions.

\begin{lemma}\label{Lemma:EvenTwists}
Following the notation of Equation \eqref{Eq:TauP}, suppose that $4$ divides $\dim F=2n$ and $\vec{L} \in T$ is such that $\tau \vec{L}$ acts trivially on $H_{n}(F)$ (which of course must be satisfied if $\tau \vec{L} = \Id_{F}$). Then $p \vec{L} \in 2\Z$.
\end{lemma}

\begin{proof}
Consider the action of a single Dehn twist $\tau_{L}$ on the determinant line $\ell = \bigwedge^{top}H_{n}(F)$. Since $n$ is even, $[L]\bullet [L]=-2$ which isn't divisible by a squared integer, implying that $[L] \in H_{n}(F)$ is primitive. Hence we can extend $[L]$ to a basis of $H_{n}(F)$ with respect to which $\tau_{L} \in \Aut(H_{n})$ can be expressed as a matrix. The first row has $-1$ in the first entry by Equation \eqref{Eq:PLZeroSection}, with the rest of the entries being zeros. In the $j$th row for $j > 1$, Equation \eqref{Eq:PicardLefschetz} determines the first entry, the $j$th entry is $1$, and all other entries are zero. Therefore $\tau_{L}$ acts on $\ell$ by $-1$. For a general $\vec{L} \in T$, $\tau \vec{L}$ acts on the determinant line by $(-1)^{p \vec{L}}$. Since we're assuming that $\tau \vec{L}$ acts trivially on $H_{n}(F)$, it must act trivially on $\ell$ and $2$ must divide $p \vec{L}$.
\end{proof}

\subsection{Some contact-topological motivations}\label{Sec:ContactMotivations}

We outline the role of positive factorizations and positive relations in modern contact topology. As stated in the introduction, our goal is to demonstrate how recent advances in the field \cite{Breen:Folded, BHH, HH:Convex} imply the existence of many positive relations on high-dimensional Weinstein manifolds. This subsection is aimed at those readers who already have some exposure to contact topology, so our exposition will be light on details.

Given a positive factorization $\vec{L}$ as above we can build a Liouville manifold $W=W_{\vec{L}}$ of $\dim W_{\vec{L}} = \dim F + 2$ as follows: Take $W_{0}$ to be $(\disk \times F, r^{2}d\theta + \beta_{F})$ with the corner $\partial \disk \times \partial F$ rounded. The $\Lambda_{j} = (e^{j2\pi i/m}, L_{j})$ will be framed Legendrian spheres in the contact boundary of $W_{0}$ (after a slight perturbation of the boundary of $W_{0}$ to manage the exact $1$-forms $\beta_{F}|_{TL_{i}}$). Then $W$ is obtained by attaching $\ind = \half \dim W_{0}$ Weinstein handles along the $\Lambda_{j}$. Such a $(W, \beta)$ is naturally the total space of a \emph{Lefschetz fibration}, as will be discussed in the next subsection. It follows that 
\begin{equation*}
\chi(W) = \chi(F) + (-1)^{\half \dim W}|\vec{L}|.
\end{equation*}
When $(F, \beta_{F})$ is Weinstein, $(W, \beta_{W})$ is as well. In any case, the boundary of $W$ is naturally a contact manifold $(\Gamma, \xi_{\Gamma})$ described as an open book with page $(F, \beta_{F})$ and monodromy $\phi$ a product of the $\tau_{L_{i}}$. This contact structure depends only on $\phi \in \pi_{0}\Symp^{c}(F, d\beta_{F})$ \cite{BHH, Giroux:ContactOB, TW:Ob}. According to \cite{BHH, GP:Lefschetz} any Weinstein manifold of $\dim \geq 4$ can be constructed in this way with $(F, \beta_{F})$ Weinstein.

It follows that each positive relation $\vec{L}, \vec{L}'$ gives us two possibly distinct Liouville fillings $(W=W_{\vec{L}}, \beta)$ and $(W' =W_{\vec{L}'}, \beta')$ of the same contact manifold $(\Gamma, \xi_{\Gamma'})$. Consequently, positive relations help us to construct distinct Liouville and Weinstein fillings of a given contact manifold, cf. \cite{Oba:Fillings, OS:InfinitelyManyFillings}. For an alternate construction of distinct fillings using relations in symplectomorphism groups, see \cite{Zhengyi:InfiniteFillings}.

Provided a positive relation $\vec{L}, \vec{L}'$, the presentation of the contact boundaries of the Liouville manifolds $W$ and $W'$ as open books with the same monodromy give us an identifications of these boundaries. As described in \cite{Breen:Folded}, we can use the fillings $W, W'$ to determine a $(-\epsilon, \epsilon)$-invariant contact structure $\xi$ on
\begin{equation*}
N = (-\epsilon, \epsilon) \times S, \quad S = W \cap_{\Gamma} W'.
\end{equation*}
We say that $S$ is a \emph{convex hypersurface} with \emph{standard neighborhood} $(N, \xi)$. Identifying $S=\{0\} \times S$, the Euler number of $\xi|_{S}$ is computed
\begin{equation}\label{Eq:ChiContact}
\int_{S} e(\xi) = \chi(W) - \chi(W') = (-1)^{\half \dim W}\left(|\vec{L}| - |\vec{L}'|\right).
\end{equation}

Now we use these contact-topological techniques to establish the existence many positive relations.

\begin{thm}\label{Thm:TwistLength}
Let $n \in \Z_{\geq 2}$ and let $m$ be an integer which we require to be even when $n$ is odd. Then there is a Weinstein $(F, \beta_{F})$ of $\dim=2n - 2$ having positive relation $\vec{L},\vec{L}'$ with $|\vec{L}| - |\vec{L}'|=m$.
\end{thm}

Clearly whenever $m \neq 0$, the positive relation $\vec{L},\vec{L}'$ cannot be braid-like. Note that the condition on the parity of $m$ cannot be removed by Lemma \ref{Lemma:EvenTwists}. To prove the theorem, we start with an abstract construction.

\begin{ex}\label{Ex:ACToConvex}
Let $(X, J)$ be an almost complex manifold of $\dim X = 2n$ and let $\xi_{X} = TX$ be the associated almost contact structure on $M = \Circle \times X$. By \cite{BEM:OT} -- which includes a definition of almost contact structures -- there is a (genuine) contact structure $\xi$ on $M$ which is homotopic to $\xi_{X}$ through almost contact structures. By \cite{HH:Convex}, we may apply a $\mathcal{C}^{0}$ small perturbation to $\{0\} \times X \subset M$ to obtain an orientable convex hypersuface $S \subset M$. Since $TS = \R \oplus \xi$, the restriction of $\xi|_{\hypersurface}$ gives $(S, \xi)$ the structure of a stably almost complex manifold. The fact that $\xi_{X}$ and $\xi$ are homotopic implies that the cobordism class $(S, \xi)$ agrees with that of $(X, J)$ in the (stable) complex cobordism ring $\Omega^{U}_{2n}$. In particular
\begin{equation*}
\int_{S} e(\xi) = \int_{S} c_{n}(\xi) = \int_{X} c_{n}(TX) = \int_{X} e(TX) = \chi(X).
\end{equation*}
\end{ex}

\begin{lemma}
Every class $\eta \in \Omega^{U}_{2n}$ can be represented by a connected convex hypersurface when $n \geq 2$.
\end{lemma}

\begin{proof}
To start we note that every $\eta$ can be represented as a disjoint union $(X, J)$ of algebraic varieties, cf. \cite[\S VII]{Stong:Cobordism}. Apply the construction of the previous paragraph to such a possibly disconnected $X$, obtaining a $(S_{0}, \xi_{0})$ in some (possibly disconnected) $(M_{0}, \xi_{0})$ representing the bordism class $\eta$. After applying contact connected sums to the components of $(M_{0}, \xi_{0})$, we obtain a connected $\Mxi$ and a connected convex hypersurface representative $\hypersurface \subset \Mxi$ of the homology class $[S_{0}] \in H_{2n}(M)$ will be as desired.
\end{proof}

\begin{proof}[Proof of Theorem \ref{Thm:TwistLength}]
Let $n,m$ be as in the statement of the theorem. Choose some $\eta \in \Omega^{U}_{2n}$ with $c_{n}(\eta) = m$. Explicitly, one could take
\begin{equation*}
\eta = \begin{cases}
m[\Proj^{n}] - \frac{m}{2}[\ProjOne \times \Proj^{n-1}] & n\ \text{odd}, \\
m(n-1)[\Proj^{n}] - \frac{mn}{2}[\Proj^{1}\times \Proj^{n-1}] & n\ \text{even}.
\end{cases}
\end{equation*}

Applying the preceding lemma, we let $(S, \xi)$ be a connected convex hypersurface representing the class $\eta$. By \cite[Theorem 1.8]{Breen:Folded} there is some Weinstein $(F, \beta_{F})$ of $\dim=2n-2$ having a positive relation $\vec{L}, \vec{L}'$ determining $(\hypersurface, \xi)$. By Equation \eqref{Eq:ChiContact}, $\int_{S}e(\xi) = m$ determines the difference in the lengths of the positive factorizations.
\end{proof}

As convex hypersurfaces $S$ are only $\mathcal{C}^{0}$-dense, rather than $\mathcal{C}^{1}$-dense in $\dim \geq 5$ contact manifolds (cf. \cite{Chaidez:Blender}), it is very difficult to explicitly describe such convex surface neighborhoods $(N \supset S, \xi)$, even for convex hypersurfaces in well-studied $\Mxi$. In searching for symplectic mapping class relations, we seek to have some interesting examples of convex hypersurfaces at least for which the $(F, \beta_{F})$ can be described somewhat explicitly (e.g. as affine or quasi-projective varieties). In a follow-up article, we will study the $(N, \xi)$ determined by some of the positive relations in this article using the techniques of \cite{Avdek:Convex} and establish that each $\eta \in \Omega^{U}_{2n}$ can in fact be represented by a convex hypersurface with a tight neighborhood.

\subsection{Review of algebraic Lefschetz pencils}\label{Sec:AlgebraicPencilReview}

Let $\Line$ be a very ample line bundle on a smooth, closed variety $X$. By definition, this means that a $\C$-basis $(\multisec_{i})_{i=0}^{N}$ of $H^{0}(X, \Line)$, the space of holomorphic sections of $\Line \to X$, determines a holomorphic embedding
\begin{equation}\label{Eq:LineBundleProjectiveEmbedding}
\phi_{\Line}: X \rightarrow \Proj^{N}, \quad \phi_{\Line}(z) = [\multisec_{0}(z):\cdots:\multisec_{N}(z)], \quad N=\dim_{\C}H_{0}(X, \Line)-1.
\end{equation}
A change of basis corresponds to composing $\phi_{\Line}$ with a linear automorphism of $\Proj^{N}$, which is a biholomorphism. Since the complement of each hyperplane $\{ z_{i} = 0\} \simeq \Proj^{N-1} \subset \Proj^{N}$ is affine, we see that $\phi_{\Line}$ restricted to the complement of each
\begin{equation*}
Z_{i} = \{ \multisec_{i} = 0\} = (\{ z_{i} = 0\}\cap X) \subset X
\end{equation*}
gives an embedding into $\C^{N}$. So each $X \setminus \{ \multisec_{i} = 0\}$ is naturally an affine variety. Replacing $\Line$ with $\Line^{\otimes k} = \phi_{\Line}^{\ast}\bigO_{k}$ for $k \in \Z_{>0}$ gives more very ample line bundles on $X$. Since $\phi_{\Line}$ is holomorphic, 
\begin{eqnarray*}
\omega_{\Line} = \phi_{\Line}^{\ast}\omega_{\Proj^{N}} \in \Omega^{2}(X), \quad \omega_{\Proj^{N}} = \frac{i}{2\pi}\partial \delbar \log(\norm{z}^{2}) \in \Omega^{1}(\Proj^{N})
\end{eqnarray*}
is a symplectic form on $X$ with respect to which we can view $\phi_{\Line}$ as a symplectic embedding. This $\omega_{\Proj^{N}}$ is the Fubini-Study form, normalized to have volume $1$.

A pair $(\multisec_{0}, \multisec_{1})$ of distinct holomorphic sections of $\Line$ yields
\be
\item subvarieties $Z_{i} \subset X$ of $\dim_{\C}Z_{i} = \dim_{\C}X - 1$ when the sections $\multisec_{i}$ are transversely cut out,
\item a \emph{base locus} $B = Z_{0} \cap Z_{1}$ of $\dim_{\C}B = \dim_{\C}X - 2$ when $\multisec_{0}\oplus \multisec_{1}$ is transversely cut out, and
\item a holomorphic function $\pi = \multisec_{1}/\multisec_{0}: X \setminus Z_{0} \rightarrow \C$.
\ee
We say that $(\multisec_{0}, \multisec_{1})$ is a \emph{Lefschetz pencil for $\Line$} if the following criteria are met:
\be
\item The $Z_{i}$ are transversely cut out, and hence smooth with $\PD_{X}(\omega_{\Line}) = [Z_{i}] \in H_{\dim X-2}(X)$,
\item The $Z_{i}$ are transverse to one another, so $B$ is smooth with $\PD_{Z_{i}}(\omega_{\Line}) = [B] \in H_{\dim X-4}(Z_{i})$.
\item The function $\pi$ has only non-degenerate critical points $\Crit(\pi) \subset X \setminus Z_{0}$ with pair-wise distinct critical values in $\Critv(\pi) \subset \C$.
\ee

We remark that the transversality requirements in this definition are non-standard. It follows from Bertini's theorem that a generic pair of sections of an ample bundle will determine a pencil which extend to a $\phi_{\Line}$ as above. See, for example, \cite[\S 1.1]{GriffithsHarris} for a review. The proof of the following lemma is well-known. We give a quick sketch since variations of the proof will be used throughout the article. 

\begin{lemma}\label{Lemma:SympHolonomy}
The symplectomorphism type of the $(Z_{1}, \omega_{\Line}, B)$ associated to a pencil for $\Line$ is independent of sections $\multisec_{0}, \multisec_{1}$.
\end{lemma}

\begin{proof}
Consider a generic $T \in [0,1]$ family $(\multisec_{0}^{T}, \multisec_{1}^{T})$ of pairs of sections of $\Line$. By Bertini's theorem the set of such pairs for which $(\multisec_{0}, \multisec_{1})$ fails to be a Lefschetz-pencil has complex codimension $\geq 1$, that is, real codimension $\geq 2$. Therefore we can assume that the path of pairs of sections determines a path of pencils. Applying a Darboux-Moser-Weinstein argument \cite[\S 3.2]{MS:SymplecticIntro} to the $T$-family establishes the claim.
\end{proof}

When $(\multisec_{0}, \multisec_{1})$ is a Lefschetz pencil then $\pi$ is a \emph{Lefschetz fibration}. We will outline the essential properties and point to the references, e.g. \cite{AS:Pencil, Donaldson:Lefschetz, OS:SurgeryBook, Seidel:Book, Torricelli:Projective}, for a full definition. The complement
\begin{equation*}
W = X \setminus Z_{0}
\end{equation*}
is a smooth affine variety as described above. Therefore $W$ inherits a Liouville structure from 
\begin{equation*}
\beta=-\frac{i}{2\pi}dh\circ J_{0} \in \Omega^{1}(\C^{N}), \quad h = \log(\norm{z}^{2} + 1)
\end{equation*}
A computation yields $d\beta = \omega_{\Proj^{N}}$ in an affine chart so that $d\beta|_{TW} = \omega_{\Line}$. The plurisubharmonic function $h$ provided, in fact, presents $\C^{N}$ and $W$ as \emph{Stein manifolds}, determining Weinstein manifold structures \cite{SteinToWeinstein}. The regular fibers
\begin{equation*}
F_{z} = \pi^{-1}(z), \quad z \in \C \setminus \Critv(\pi)
\end{equation*}
of $\pi$ are then also smooth, affine varieties and hence naturally Weinstein. The closure of each $F_{z}$ in the closure $X$ of $W$ intersects $Z_{0}$ along $B$. By considering lifts of paths in $\C \setminus \Critv(\pi)$ via $\pi^{-1}$, we see that each $(F_{z}, \beta)$ is deformation equivalent to
\begin{equation*}
(F, \beta), \quad F := F_{0} = Z_{1} \setminus B.
\end{equation*}

The pencil $(\multisec_{0}, \multisec_{1})$ determines a compactly supported symplectomorphism of $(F, d\beta)$ in two ways. One way to express this symplectomorphism comes from studying the monodromy of $\pi$. Let $\eta_{i}$ be a collection of embedded, oriented, $[0,1]$-parameterized paths in $\C$ indexed by the $z_{i} \in \Critv(\pi)$, beginning at $0 \in \C$, ending at $z_{i}$, and disjoint away from $0 \in \C$. We require that the indices $i$ are ordered so that tangent vectors of the $\eta_{i}$ at $0$ are ordered counterclockwise. Each $\eta_{i}$ determines a Lagrangian disk $D_{i}$ in $W$ for which $\pi D_{i} = \eta$ and for which $L_{i} = \partial D_{i}$ is an exact Lagrangian sphere in $F$. The $\eta_{i}$ are \emph{vanishing paths} and the $D_{i}$ are their associated \emph{Lefschetz thimbles} which provide framings of the $L_{i}$. Let $\gamma \subset \C \setminus \Critv(\pi)$ be a parameterized simple closed curve based at $0$ and oriented counterclockwise, otherwise disjoint from the $\eta_{i}$, encircling all of the $\eta_{i}$, and encircling $\eta_{1}$ first. Then the monodromy morphism $\Phi_{\gamma} \in \Symp^{c}(F, \omega_{\Line})$ can be expressed as a product of Dehn twists
\begin{equation*}
\Phi_{\gamma} = \tau_{L_{\# \Crit(\pi)}} \cdots \tau_{L_{1}} \in \pi_{0}\Symp^{c}(F, \omega),
\end{equation*}
generalizing \cite{Arnold:Singularity}. Modifying the choice of paths $(\eta_{i})$ applies conjugacy relations to $\vec{L} = (L_{i})$.

Second, note that a collar neighborhood of the ideal boundary $\Gamma$ of $F$ is identical to a $\disk$-bundle neighborhood of $B \subset Z_{1}$ with its zero section removed. We can associate to this data a \emph{fibered Dehn twist} $\tau_{\partial F} \in \pi_{0}\Symp^{c}(F, \omega_{\Line})$, cf. \cite{CDvK, Oba:FourDRelation}. We give a brief review: Our collar neighborhood of $F$ with a neighborhood of $B$ removed is given as a finite symplectization $([0, 1]_{s}\times \Gamma, d(e^{s}\alpha))$. Here $\alpha$ is a contact form on $\Gamma$ whose Reeb flow $R_{\alpha}$ is $2\pi$-periodic, generating circular rotation on the disk bundle. Using this collar we can write
\begin{equation*}
\tau_{\partial F}(s, x) = \left(s, \Flow_{R_{\alpha}}^{g(s)}(x)\right)
\end{equation*}
for a smooth function $g = g(s)$ which is decreasing with $g(0) = 2\pi$ and $g(s) = 0$ for $s \geq 1$. Then $\tau_{\partial F}$ extends to the rest of $F$ as the identity map. The induced action on $H_{\ast}(F)$ by $\tau_{\partial F}$ is then clearly the identity.

The first statement of the following theorem is proved in \cite{Aroux:Pencil, Gompf:Pencil}, with special cases worked out in \cite{AA16, Oba:FourDRelation}. The count of critical points as a weighted sum of Euler characteristics follows easily from the Lefschetz hyperplane theorem, cf. \cite[Theorem 4.5]{Oba:FourDRelation}.

\begin{thm}\label{Thm:BoundaryRelation}
In the above notation
\begin{equation*}
\tau_{L_{|\vec{L}|}} \cdots \tau_{L_{1}} = \tau_{\partial F} \in \pi_{0}\Symp^{c}(F, \omega_{\Line})
\end{equation*}
In particular, $\tau_{\partial F}$ is isotopic to a symplectomorphism determined by a product of
\begin{equation*}
|\vec{L}| = \# \Crit\pi = (-1)^{\dim_{\C}X}\left(\chi(X) - 2\chi(Z_{1}) + \chi(B)\right)
\end{equation*}
positive Dehn twists along Lagrangian spheres.
\end{thm}

Note that $\tau_{\partial F}$ can be defined on any Liouville manifold $(F, \beta_{F})$ whose ideal contact boundary is a \emph{Boothby-Wang circle bundle} over some $(B, \omega)$. See \cite{CDvK} for further details. In this general setting, not all such $\tau_{\partial F}$ can be factorized as products of Dehn twists along exact Lagrangian spheres. For example, take $(F, \beta_{F})$ to be the complement of a $\deg \geq 2$ hypersurface in $\Proj^{n}, n>1$. The ideal boundary of $F$ is of Boothby-Wang type, so that $\tau_{\partial F}$ is defined but such $(F, d\beta_{F})$ contain no Lagrangian spheres since $\Proj^{n}$ contains no Lagrangian spheres. This is obvious for homological reasons when $n$ is even and non-trivially follows from \cite{Seidel:Graded} for $n$ is odd. When $\deg =2$, these $(F, d\beta_{F})$ coincides with $(\disk^{\ast}\mathbb{RP}^{n}, dp\wedge dq)$ \cite{Biran:LagrangianBarriers}. See \cite{Audin:LagrangianSkeletons, Torricelli:Projective} for generalizations to cotangent disk bundles of other symmetric spaces. A general criteria for finding fibered Dehn twists which cannot by factored into products of Dehn twists along Lagrangian spheres is given in \cite{CDvK} (attributed to Biran and Giroux). For a generalization of the definition of fibered Dehn twist provided above, see \cite{Perutz:Matching, Torricelli:Projective}.

\section{Symplectic mapping class relations from pencil pairs}\label{Sec:PencilPairs}

In this section we define pencil pairs and then appeal to Theorem \ref{Thm:BoundaryRelation} to show how they determine positive relations (Definition \ref{Def:PositiveRelation}). Then we will describe some low-dimensional examples.

\subsection{Pencil pairs}

\begin{defn}
Let $X^{\pm}$ be closed varieties equipped with very ample line bundles $\Line^{\pm}$, Lefschetz pencils $(\multisec_{0}^{\pm}, \multisec_{1}^{\pm})$ for the $\Line^{\pm}$, and associated subvarieties $B^{\pm} \subset Z^{\pm}_{i} \subset X^{\pm}$. If there is a symplectomorphism
\begin{equation*}
\phi: (Z_{1}^{+}, \omega_{\Line^{+}}) \rightarrow (Z_{1}^{-}, \omega_{\Line^{-}}), \quad \phi(B^{+})=B^{-},
\end{equation*}
we say that $(X^{\pm}, \Line^{\pm},\phi)$ is a \emph{pencil pair}.
\end{defn}

It is implicit in our notation that the pencil pair depends only on the $\Line^{\pm}$ rather than the specified sections $(\multisec_{0}^{\pm}, \multisec_{1}^{\pm})$. This abuse of notation is permitted by Lemma \ref{Lemma:SympHolonomy}. 

A pencil pair together with choices of vanishing paths associated to the $\pi^{\pm} = \multisec^{\pm}_{1}/\multisec^{\pm}_{0}$ determines a mapping class relation between products of positive Dehn twists on the Weinstein manifold
\begin{equation*}
(F, \omega_{F}) = \left(Z^{+}_{1} \setminus B, \omega^{+}|_{TZ_{1}^{+}}\right)
\end{equation*}
as follows: We have a collection $\vec{L}^{+} = (L^{+}_{i})_{i=1}^{\# \Crit \pi^{+}}$ of Lagrangian spheres in $(F, \omega_{F})$ determined by the holomorphic function $\pi^{+}=\multisec^{+}_{1}/\multisec^{+}_{0}$ and a choice of matching paths for $\Critv \pi^{+}$. Here we recall that $F$ is a smooth fiber of $\pi^{+}$. We also have a collection $\phi^{-}\vec{L}^{-} = (\phi^{-1}L^{-}_{i})_{i=1}^{\# \Crit \pi^{-}}$ of Lagrangian spheres in $(F, \omega_{F})$ determined by $\pi^{-}$ and a choice of matching paths for $\Critv \pi^{-}$. Then by Theorem \ref{Thm:BoundaryRelation},
\begin{equation*}
\tau_{L_{|\vec{L}^{+}|}^{+}} \cdots \tau_{L_{1}^{+}} = \tau_{\phi^{-1}L_{|\vec{L}^{-}| }^{-}} \cdots \tau_{\phi^{-}L_{1}^{-}} \in \pi_{0}\Symp_{0}(F, \omega_{F}).
\end{equation*}
In other words, the pair $\vec{L}^{+}, \phi^{-1}\vec{L}^{-}$ is a positive relation. Theorem \ref{Thm:BoundaryRelation} also yields
\begin{equation*}
|\vec{L}^{+}| - |\vec{L}^{-}| = \# \Crit \pi^{+} - \#\Crit \pi^{-} = (-1)^{\dim_{\C}X^{\pm}}(\chi(X^{+}) - \chi(X^{-})).
\end{equation*}
In light of the results of \S \ref{Sec:Braids}, this formula tells us that whenever we have a pencil pair with $\chi(X^{+}) \neq \chi(X^{-})$, then the associated mapping class relation is non-braid-like.

\subsection{The case $\dim_{\C}X^{\pm} =2$}

Suppose that $X$ is a closed variety of $\dim_{\C}X=2$ equipped with a very ample line bundle $\Line$ having a pencil with associated holomorphic map $\pi$. Then our divisor $Z_{1}$ is a Riemann surface and base locus $B \subset Z_{1}$ is a set of points. Applying Poincar\'{e} duality and the adjunction formula,
\begin{equation}\label{Eq:4dChernNumbers}
\begin{gathered}
\# B = [Z_{1}]^{2} = \int_{Z_{1}}\omega_{\Line} = \int_{X}c_{1}(\Line)^{2},\\
\chi(Z_{1}) = \int_{Z_{1}}c_{1}(X) - [Z_{1}]^{2} = \int_{X}(c_{1}(X) - c_{1}(\Line))\wedge c_{1}(\Line)\\
\implies \# \Crit \pi = \int_{X}c_{2}(X) + (c_{1}(\Line) - 2c_{1}(X))\wedge c_{1}(\Line).
\end{gathered}
\end{equation}
Since the symplectomorphism type of $(Z_{1}, \omega_{\Line}, B)$ is then determined by $\chi(Z_{1})$, the symplectic area of $Z_{1}$, and $\# B$ (which is redundant), the following theorem is obvious.

\begin{lemma}\label{Lemma:LowDimPencil}
Let $X^{\pm}$ be a pair of $\dim_{\C}=2$ varieties equipped with very ample line bundles $\Line^{\pm} \to X^{\pm}$. Then there is a $\phi$ for which $(X^{\pm}, \Line^{\pm}, \phi)$ is a pencil pair iff the quantities
\begin{equation*}
\int_{X^{\pm}}c_{1}(\Line^{\pm})^{2}, \quad \int_{X^{\pm}}c_{1}(\Line^{\pm})c_{1}(X^{\pm})
\end{equation*}
are independent of their $\pm$ signs.
\end{lemma}

Using the fact that $c_{1}(\Line^{\otimes k}) = kc_{1}(\Line)$, the following lemma follows easily from Equation \eqref{Eq:4dChernNumbers}.

\begin{lemma}\label{Lemma:LowDimCabling}
Let $(X^{\pm}, \Line^{\pm}, \phi)$ be a pencil pair with $\dim_{\C}X^{\pm} = 2$. Then for all $k \geq 1$, there are $\phi_{k}$ for which the $(X^{\pm}, (\Line^{\pm})^{\otimes k}, \phi_{k})$ are pencil pairs.
\end{lemma}

\subsection{Generalized lantern relations}\label{Sec:GeneralizedLantern}

The following family of examples is well-known: Consider
\begin{equation*}
X^{+}=\Proj^{2},\quad  c_{1}(\Line^{+})=2\omega_{\Proj^{2}}, \quad X^{-} = \ProjOne \times \ProjOne, \quad c_{1}(\Line^{-}) = (2\pi_{1}^{\ast} + \pi_{2}^{\ast})\omega_{\ProjOne}.
\end{equation*}
Here the $\pi_{i}$ denote projections onto the $\ProjOne$ factors of $X^{-}$. A quick calculation shows that these $(X^{\pm}, \Line^{\pm})$ form a pencil pair by Lemma \ref{Lemma:LowDimPencil}. The $Z^{\pm}_{1}$ are $\ProjOne$s each having self-intersection number $4$. The contact $3$-manifold given by the boundary of a tubular neighborhood of a $Z^{\pm}_{1} \subset X^{\pm}$ is the standard lens space $(L(4, 1), \xi_{std})$ given by reducing the standard contact $3$-sphere $\Sthree = \partial(\disk^{4}, \beta_{0})$ by the $\Z/4\Z$ action generated by $z \mapsto iz$ for $z \in \sphere^{3} \subset \C^{2}$. The $W^{\pm} = X^{\pm} \setminus Z^{\pm}_{1}$ are its unique symplectic fillings (modulo symplectic blow-up) according to McDuff \cite{McDuff:RationalRuled}. The Lefschetz fibration on $W^{+}$ determined by a holomorphic Lefschetz pencil has fiber a $4$-punctured sphere with three critical points. As noted after Theorem \ref{Thm:BoundaryRelation}, $W^{+}$ is symplectomorphic to a $\disk^{\ast}\mathbb{RP}^{2}$. The Lefschetz fibration on $W^{-}$ has fiber a $4$-punctured sphere with four critical points. As observed by Auroux and Smith \cite[\S 5.2]{AS:Pencil} the Lefschetz fibration on $W^{+}$ has monodromy given the left hand side of the lantern relation as it appears in Figure \ref{Fig:Lantern}. An analysis similar to that of \cite{AS:Pencil} reveals that the monodromy of the Lefschetz fibration on $W^{-}$ is given by the right-hand side of the lantern relation as it appears in Figure \ref{Fig:Lantern}. See \cite[\S 5]{Torricelli:Projective}. Lemma \ref{Lemma:LowDimCabling} yields pencil pairs for the the $(X^{\pm}, (\Line^{\pm})^{\otimes k})$ for each $k \in \Z_{>0}$. We can view these pencil pairs as generalizing the lantern relation, since for any $k$ the difference in the number of critical points is always $\chi(X^{-}) - \chi(X^{+}) = 1$.

\begin{rmk}
As we learned from reading \cite{Luo:Relations}, the lantern relation was first found by Dehn \cite[p.333]{Dehn:CollectedPapers} and then later rediscovered by Johnson \cite{Johnson:Lantern}, to whom it is often attributed.

An alternative, higher-dimensional generalization of the lantern relation (associated to the total space of $\disk^{\ast}\Proj^{2}$ rather than $\disk^{\ast}\mathbb{RP}^{2}$) is described in \cite{Torricelli:Projective}.
\end{rmk}

Now we describe how the above $k \in \Z_{> 0}$ family fits into a larger collection of examples starting with a general construction. Let $\Sigma$ be a closed Riemann surface of Euler characteristic $\chi$, choose some integer $d > 1 + g_{\Sigma} = 3-\chi$, and let $\ell_{\chi, d} \to \Sigma$ be a line bundle of $\deg \ell_{\chi, d} = -d$. Therefore the dual line $\ell_{\chi, d}^{\ast} \to \Sigma$ has $\deg = d$. Write $X_{\chi, d} = \mathbb{P}(\C \oplus \ell_{\chi, d})$ for the projectivization, which is a complex ruled surface,
\begin{equation*}
\ProjOne \to X_{\chi, d} \xrightarrow{p} \Sigma.
\end{equation*}
See for example \cite[Example 4.26]{MS:SymplecticIntro}. Since $X_{\chi, d}$ is a $2$-sphere bundle with a section it follows, e.g. from a Gysin sequence, that $\chi(X_{\chi, d}) = 2\chi$. The rank $2$ bundle $(\C \oplus \ell_{\chi, d})^{\ast} \simeq \C \oplus \ell_{\chi, d}^{\ast}$ over the total space of $\C \oplus \ell_{\chi, d}$ descends to a line bundle $\Line_{\chi, d} \to X_{\chi, d}$. Sections $\multisec \in H^{0}(X_{\chi, d}, \Line_{\chi, d})$ of $\Line_{\chi, d}$ have the form
\begin{equation*}
\multisec[z_{0}, z_{1}] = az_{0} - \multisec^{\ast}(z_{1}), \quad a, z_{0} \in \C, \quad z_{1} \in \ell_{\chi, d}, \quad \multisec^{\ast} \in H^{0}(\Sigma, \ell_{\chi, d}^{\ast}).
\end{equation*}
Since $\ell_{\chi, d}^{\ast}$ is very ample, so is $\Line_{\chi, d}$. A section $\multisec_{si}$ of $\Line_{\chi, d}$ for which $a=0$ and $\multisec^{\ast}$ is generic will have singular zero locus $Z_{si}$ whose irreducible components are the $[1:0]$ submanifold together with $d$ fibers of $p$. Take a generic $a \neq 0$ section $\multisec_{sm}$ close to $\multisec_{si}$ whose zero locus $Z_{sm}$ is smooth. Then $Z_{sm}$ is obtained by smoothing the singularities of $Z_{si}$, so $\chi(Z_{sm}) = \chi$ and $\PD_{X_{\chi, d}}[Z_{sm}] = c_{1}(\Line_{\chi, d})$. Analyzing the intersections of $Z_{sm}$ and $Z_{si}$ (which have the same $H_{2}(X_{\chi, d})$ class) yields $\int_{X_{\chi, d}} c_{1}(\Line_{\chi, d})^{2} = d$. The adjunction formula applied to $Z_{sm}$ gives $c_{1}(TX_{\chi, d}) = \chi + d$. Combining our calculations with Equation \eqref{Eq:4dChernNumbers}, a Lefschetz fibration associated to a pencil $(\multisec_{0}, \multisec_{1})$ for $\Line_{\chi, d}^{\otimes k} \to X_{\chi, d}$ has
\be
\item $Z_{0, \chi, d, k} = \multisec_{0}^{-1}(0)$ with $[Z_{0, \chi, d, k}] = dk^{2}$ and $\chi(Z_{0, \chi, d, k}) = k(\chi + d - dk)$,
\item total space $W_{\chi, d, k} = X_{\chi, d} \setminus Z_{0, \chi, d, k}$ with $\chi(W_{\chi, d, k}) = 2\chi - k\chi - kd + dk^{2}$,
\item $\# \Crit = 3dk^{2} + 2(\chi - kd - k\chi)$ critical points of $\pi = \multisec_{1}/\multisec_{0}: X_{\chi, d, k} \to \C$, and
\item non-singular fiber of $\pi$ of genus $g_{\chi, d, k} = 1 + \frac{k}{2}(\chi + d-dk)$ with $dk^{2}$ punctures.
\ee

Note that $X_{2,2} \to \ProjOne$ is a trivial fibration, and the symplectic manifold we get is the second Hirzebruch surface, identical to $(X^{-}, \omega_{\Line^{-}})$ above. A quick calculation shows that for any $m, k > 0$, $(\Proj^{2}, \bigO_{mk})$ and $(X_{\chi, m^{2}}, \Line_{\chi, m^{2}}^{\otimes k})$, $\chi=m(3-m)$ are a pencil pair via Lemma \ref{Lemma:LowDimPencil}.

\begin{thm}\label{Thm:PtwoPonePoneMappingClass}
Let $F_{m,k}$ be an oriented surface of genus $g_{m,k} = 1+\half mk(mk-3)$ with $(mk)^{2}$ boundary components. Writing $\chi=m(3-m)$, a boundary Dehn twist $\tau_{\partial F_{m,k}}$ can be factorized as products of
\be
\item $(mk)^{2}$ Dehn twists along the individual boundary components of $F_{m,k}$ determined by $\tau_{\partial F_{m,k}}$,
\item $3(mk)^{2} - 2m(3k + m - 3)$ Dehn twists determined by a pencil for $(X_{\chi, m^{2}}, \Line_{\chi, m^{2}}^{\otimes k})$, and
\item $3((mk)^{2} - 2mk + 1)$ Dehn twists determined by a pencil $(\Proj^{2}, \bigO_{mk})$.
\ee
For the $k=1, m=2$ case the first two collections of curves coincide, and the collections of curves from the last two items give the classical lantern relation.
\end{thm}

By factoring $mk$ in different ways, the above constructions can give many holomorphic Weinstein fillings $W_{\chi, m^{2}, k}, \Proj^{2} \setminus \Sigma_{mk}$ of the contact $3$-manifold given as the ideal boundary of the complement of a $\deg = mk$ curve $\Sigma_{mk} \subset \Proj^{2}$. For example, choose $N \in \Z_{\geq 2}$ and $m = 2^{i}, k = 2^{N-i}$ for $0 < i < N$. Then
\begin{equation*}
\chi(W_{\chi, 2^{2i}, 2^{N-i}}) = 2^{i+1}(3-2^{i}) - 2^{N + 1}, \quad \chi(\Proj^{2} \setminus \Sigma_{2^{N}}) = 3 + 2^{N}(2^{N} - 3).
\end{equation*}
These Euler characteristics are all distinct, implying the following theorem.

\begin{thm}
Let $\Mxi$ be the ideal contact boundary of a tubular neighborhood of a $\deg=2^{N}$ curve in $\Proj^{2}$ with $N \geq 2$. Then $\Mxi$ has at least $N$ distinct (holomorphic) Weinstein fillings.
\end{thm}

\section{Pencil pairs from projective models and high-degree hypersurfaces}\label{Sec:ProjectiveModel}

The existence of our $\dim_{\C} X^{\pm} = 3$ pencil pairs (determining mapping class relations on Weinstein $4$-manifolds) relies on identifications of divisors $Z^{\pm}_{1} \subset X^{\pm}$ with complete intersections in ``big'' projective varieties $\ModSpace$. In \S \ref{Sec:ProjectiveModelDetails} we describe this technique. In \S \ref{Sec:Cabling} we show that a single pencil pair found using this technique determines an infinite family of pencil pairs, generalizing Lemma \ref{Lemma:LowDimCabling} to all dimensions.

\subsection{Projective models}\label{Sec:ProjectiveModelDetails}

\begin{defn}\label{Def:ProjectiveModel}
Let $(Z, \Line)$ be a polarized variety of $\dim_{\C}Z = n$, $\ModSpace \subset \Proj^{N}$ be a projective variety of $\dim_{\C}\ModSpace = m > n$, $\vec{d} = (d_{1}, \dots, d_{m-n}) \in \Z^{m-n}_{>0}$, and write $\bigO_{\vec{d}} = \oplus_{j=1}^{m-n} \bigO_{d_{j}}$. We say that $(\ModSpace, \vec{d})$ is a \emph{projective model} for $(Z, \Line)$ if there is a holomorphic section of $\multisec_{\ModSpace}$ of $\bigO_{\vec{d}} \to \Proj^{N}$ and a biholomorphism
\begin{equation*}
\phi: Z \rightarrow Z_{\ModSpace}, \quad Z_{\ModSpace} = \{\multisec_{\ModSpace}=0\}\cap \ModSpace
\end{equation*}
for which $\phi^{\ast}\bigO_{d_{1}} = \Line$.
\end{defn}

\begin{defn}\label{Def:CommonProjectiveModel}
Let $(X^{\pm}, \Line^{\pm})$ be a pair of very ample polarized varieties. Suppose we have holomorphic sections $\multisec^{\pm}_{1} \in H^{0}(X^{\pm}, \Line^{\pm})$ for which
\be
\item the zero loci $Z_{1}^{\pm} = \{ \multisec^{\pm}_{1} = 0\} \subset X^{\pm}$ are transversely cut out, and
\item the polarized varieties $(Z_{1}^{\pm}, \Line^{\pm})$ have a common projective model $(\ModSpace, \vec{d})$.
\ee
Then we say that the $(X^{\pm}, \Line^{\pm})$ \emph{have divisors with a common projective model}.
\end{defn}

\begin{lemma}\label{Lemma:ProjectiveModel}
Suppose that the $(X^{\pm}, \Line^{\pm})$ have divisors with a common projective model. Then there are pencils for the $\Line^{\pm} \rightarrow X^{\pm}$ and a $\phi$ for which $(X^{\pm}, \Line^{\pm}, \phi)$ is a pencil pair.
\end{lemma}

\begin{proof}
By the definition of projective models there are sections $\multisec^{\pm}_{\ModSpace, 1}$ of $\bigO_{\vec{d}}$ with zero loci $Z_{\ModSpace}^{\pm} \subset \ModSpace$ and biholomorphisms $\phi^{\pm}: Z^{\pm}_{1} \rightarrow Z_{\ModSpace}^{\pm}$ for which $\Line^{\pm}|_{Z_{1}^{\pm}} = \phi^{\ast}\bigO_{d_{1}}$. Choose  sections $\multisec^{\pm}_{0}$ of the $\Line^{\pm}$ so that the $(\multisec^{\pm}_{0}, \multisec^{\pm}_{1})$ are pencils for the $(X^{\pm}, \Line^{\pm})$ with base loci
\begin{equation*}
B^{\pm}=\{ \multisec^{\pm}_{0}=\multisec^{\pm}_{1}=0\} \subset Z^{\pm}_{1}.
\end{equation*}
By the existence of the projective models, there are holomorphic sections $\multisec_{\ModSpace, 0}^{\pm}$ of $\bigO_{d_{1}} \to \Proj^{N}$ such that
\begin{equation*}
\phi^{\pm}B^{\pm}=\{\multisec^{\pm}_{\ModSpace, 1} \oplus \multisec^{\pm}_{\ModSpace,0}=0\} \cap \ModSpace.
\end{equation*}

Let $\multisec^{T}_{\ModSpace, 1} \oplus \multisec^{T}_{\ModSpace,0}$ be a $T \in [-1, 1]$ parameter family of holomorphic sections of $\bigO_{\vec{d}} \oplus \bigO_{d_{1}} \rightarrow \Proj^{N}$ such that $\multisec^{\pm 1}_{\ModSpace, j} = \multisec_{\ModSpace, j}^{\pm}$ for $j=0,1$ and such that the zero loci of the $\multisec^{T}_{\ModSpace, 1}$ and $\multisec^{T}_{\ModSpace}\oplus \multisec^{T}_{\ModSpace, 0}$ are transversely cut out for all $T$. The transversality conditions can be satisfied using the fact that failure of transversality is a $\codim_{\C} \geq 1$ condition in the projective spaces of sections $\Proj(H^{0}(\ModSpace, \bigO_{\vec{d}}\oplus\bigO_{d_{1}}))$. A parametric Darboux-Moser-Weinstein argument as in the proof of Lemma \ref{Lemma:SympHolonomy} gives us symplectomorphism
\begin{equation*}
\psi: (\phi^{+}Z^{+}_{1}, \omega_{\bigO_{d_{1}}}) \rightarrow (\phi^{-}Z^{-}_{1}, \omega_{\bigO_{d_{1}}}), \quad \psi \phi^{+}B^{+} = \phi^{-1}B^{-}.
\end{equation*}
So $\phi=(\phi^{-})^{-1}\psi\phi^{+}$ a symplectomorphism between the $(X^{\pm}, \omega_{\Line^{\pm}})$ preserving the base loci $B^{\pm}$.
\end{proof}

\subsection{Cabling pencil pairs}\label{Sec:Cabling}

Theorem \ref{Theorem:Cabling} below allows us to determine an infinite family of pencil pairs from a single pencil pair, when the original pencil pair is given by Lemma \ref{Lemma:ProjectiveModel}. The results can be seen as a high-dimensional generalization of Lemma \ref{Lemma:LowDimCabling}. It's proof depends on the following general lemma. We suspect that this result is known to experts but were unable to find a proof in the literature.

\begin{thm}\label{Theorem:BranchedCover}
Let $Z \subset X$ be a divisor of be a very ample line bundle $\Line$ over a closed variety $X$. View $Z$ and $X$ as symplectic manifolds with their symplectic structures inherited from $(\Proj^{N}, \omega_{\Proj^{N}}), N=\dim H^{0}(X, \Line)-1$ via the embedding of Equation \eqref{Eq:LineBundleProjectiveEmbedding}. For $k \geq 0$, let $Z_{k}$ be a divisor of $\Line^{\otimes k} = \bigO_{k}|_{X}$ intersecting $Z$ transversely. Then as a symplectic manifold, $Z_{k} \subset X \subset \Proj^{N}$ is a $k$-fold branched covering of $Z$ with branch locus $Z \cap Z_{k}$. In particular, the symplectomorphic type of the pairs $(Z_{k}, Z \cap Z_{k})$ and $(Z_{k}, B_{k})$, where $B_{k}$ is the base locus of a pencil on $X$ containing $Z_{k}$, depends only on $\Line|_{Z}$.
\end{thm}

\begin{proof}
We first explain what we'd like to do and why it wouldn't work. Assume that $Z$ is cut out by some $\multisec_{0} \in H^{0}(X, \Line)$. Take $Z_{k}$ to be cut out by a generic section $\mathcal{C}^{1}$-close to $\multisec_{0}^{\otimes k}$, so that $Z_{k}$ is contained in a tubular neighborhood $N$ of $Z$ in $X$. Then we can map $Z_{k}$ onto $Z$ by restricting a projection $\pi_{N}: N \to Z$ to $Z_{k}$. If $\pi_{N}$ is holomorphic then $\pi_{N}|_{Z_{k}}$ will be holomorphic, and we can try to prove that $\pi_{N}|_{Z_{k}}$ is a branched covering. Unfortunately, this argument does not work in general as our subvariety $Z \subset X$ is not guaranteed to have a neighborhood $N$ with a holomorphic projection $\pi_{N}$, cf. \cite{MorrowRossi}. Instead we'll deform $Z_{k} \subset X$ to a variety $Z_{k}'$ inside the total space $\totalspace(\Line|_{Z})$ of the line bundle $\Line|_{Z} \to Z$ by embedding both $X$ and $\totalspace(\Line|_{Z})$ into a large projective space. Then $Z'_{k}$ and $\totalspace(\Line|_{Z})$ have holomorphic projections to $Z$.
	
Let $\vec{1} = [1:\cdots:1] \in \Proj^{N+1}$. Viewing $\Proj^{N} = \{z_{N+1} = 0\} \subset \Proj^{N+1}$, we can think of $Z$ and $X$ as being contained in $\Proj^{N+1}$. The map
\begin{equation*}
\pi_{\vec{1}}: (\Proj^{N+1} \setminus \{\vec{1}\}) \to \Proj^{N}, \quad [z_{0}: \cdots :z_{N}: z_{N+1}] \mapsto [z_{0} - z_{N+1}:\cdots : z_{N} - z_{N+1}]
\end{equation*}
is a holomorphic submersion using which we can identify $\Proj^{N+1} \setminus \{\vec{1}\}$ with the total space of the line bundle $\bigO_{1} \to \Proj^{N}$. The spaces $\pi_{\vec{1}}^{-1}Z$ and $\pi_{\vec{1}}^{-1}X$ are then smooth, quasi-projective varieties in $\Proj^{N+1}$. They are projective cones over $Z$ and $X$, with a point $\vec{1}$ removed from each and can be smoothly compactified inside of a blow up at $\Proj^{N+1}$ at the point $\vec{1}$.

Without loss of generality, we assume that $Z \subset X$ is the zero locus of $z_{0}$ in $X$. Indeed, a basis $(\multisec_{j})$ of $H^{0}(X, \Line)$ used to define the embedding of Equation \eqref{Eq:LineBundleProjectiveEmbedding} can be chosen so that $\multisec_{0}$ is the section cutting out $Z$. Then we can express $Z$ as $\{ z_{N+1} = 0\}$ inside of $\pi_{\vec{1}}^{-1}Z$ or as $\{ z_{0} = z_{N+1}=0\}$ inside of $\pi_{\vec{1}}^{-1}X$.

For $\multisec_{k}$ a $\deg=k$ homogeneous polynomial in variables $z_{1},\cdots, z_{N}$ define
\begin{equation*}
\sigma_{N} = z_{0}^{k} + \multisec_{k} \in H^{0}(\Proj^{N}, \bigO_{k}).
\end{equation*}
We assume that $\sigma_{N}$ is transversely cut out when viewed as a section of $\bigO_{k}|_{X} = \Line^{\otimes k}$ and identify its zero locus with the subvariety of interest,
\begin{equation*}
Z_{k} = \{ \sigma_{N} = 0\} \cap X.
\end{equation*}
As a symplectic manifold we may identify $Z_{k}$ with the zero locus of any transversely cut out section of $\Line^{\otimes k} \to X$ by following the proof of Lemma \ref{Lemma:SympHolonomy}. We also assume that $Z \cap Z_{k}$ is transversely cut out by $z_{0} \oplus \sigma_{X} \in H^{0}(X, \Line \oplus \Line^{\otimes k})$. This means that $Z \cap Z_{k}$ is cut out by $\multisec_{k}$ viewed as an element of $H^{0}(Z, \Line)$.

Define a section $\sigma_{N+1}$ of $\bigO_{k}$ and a subvariety $V_{k}$ of $\pi_{\vec{1}}^{-1}(X)$ as follows:
\begin{equation*}
\sigma_{N+1} = z_{0}^{k} + \multisec_{k} + z_{N+1}^{k} \in H^{0}(\Proj^{N+1}, \bigO_{k}), \quad V_{k} = \{ \sigma_{N+1} = 0\} \cap \pi_{\vec{1}}^{-1}(X).
\end{equation*}
Applying variations to $\multisec_{k}$, we can assume that $V_{k}$ is transversely cut out of $\pi_{\vec{1}}^{-1}(X)$. Take note of the intersections of $V_{k}$ with the $z_{0}=0$ and $z_{N+1}=0$ hyperplanes:
\begin{equation*}
\begin{gathered}
Z_{k} = V_{k} \cap \{ z_{N+1} =0\} \subset X, \quad Z_{k}' := V_{k} \cap \{ z_{0}=0\} \subset \pi_{\vec{1}}^{-1}(Z),\\
Z_{k} \cap Z = Z_{k}'\cap Z = \{ \multisec_{k} = 0\}\cap Z.
\end{gathered}
\end{equation*}
We claim that both $Z_{k}$ and $Z_{k}'$ are transversely cut out of $V_{k}$. To establish transversality for $Z_{k}$ observe that along $X$, $T\pi_{\vec{1}}^{-1}(X)$ splits as a direct sum of $TX$ and the tangent spaces to the fibers of $\pi_{\vec{1}}^{-1}$. Then $d\sigma_{N+1}|_{X} = d\sigma_{N}$ is non-vanishing along the $TX$ summand and $dz_{N+1}$ is non-vanishing along the complementary summand. Transversality for $Z_{k}'$ follows a similar argument by splitting the tangent space $T\pi_{\vec{1}}^{-1}(X)$ over $Z$ as a direct sum of $T\pi_{\vec{1}}^{-1}Z$ and a normal bundle to $Z$ in $X$. It follows that a generic $1$-parameter family of sections of $\Span(z_{0}, z_{1}) \subset H^{0}(\Proj^{N+1}, \bigO_{1})$ interpolating between $z_{0}$ and $z_{N+1}$ will cut out a $1$-parameter family of subvarieties of $V_{k}$ which interpolate between $Z_{k}$ and $Z_{k}'$. We assume of course that the zero-loci of the hyperplane sections in the $1$-parameter family miss the point $\vec{1} \in \Proj^{N+1}$. Through such a $1$-parameter family, every intersection with $Z$ yields $Z_{k}\cap Z$. Again by the logic of Lemma \ref{Lemma:SympHolonomy}, the pairs $(Z_{k}, Z_{k}\cap Z)$ and $(Z_{k}', Z_{k}'\cap Z)$ are symplectomorphic. To establish that $(Z_{k}, B_{k})$ and $(Z_{k}', B_{k}')$ are symplectomorphic (with $B_{k}$ and $B_{k}'$ being the base loci of pencils containing $Z_{k}$ and $Z_{k}'$, respectively), we can take another section of $\sigma_{N+1}'$ of $\bigO_{k}$ and keep track of the intersection of the zero locus of $\sigma_{N+1}'$ with our subvarieties in the $1$-parameter family.

To complete the proof we must established that $Z'_{k} \subset \pi_{\vec{1}}^{-1}(Z)$ is a branched covering of $Z$. This subvariety is cut out of by the section
\begin{equation*}
\sigma' = \multisec_{k} + z^{k}_{N+1} \in H^{0}(\pi_{\vec{1}}^{-1}(Z), \bigO_{k}).
\end{equation*}
Identify $\pi_{\vec{1}}^{-1}(Z)$ with the total space $\totalspace(\bigO_{1}) = \totalspace(\Line)$ over $Z$ with projection map $p = \pi_{\vec{1}}$, and $\bigO_{1}$ as $p^{\ast}(\Line|_{Z})$. With respect to this identification, $z_{N+1}$ is the tautological section $T$ defined by the property that $T(v_{z}) = v_{z}$ for $v_{z} \in p^{-1}(z)$. Similarly $\multisec_{k}$ is lifted from the base $Z$ of the projection. Therefore we can view $\sigma'$ as a section of $(p^{\ast}\bigO_{1})^{\otimes k} \to \totalspace(\bigO_{1}|_{Z})$ expressible as $T^{\otimes k} + \multisec_{k}$. 

The remainder of the proof follows a standard holomorphic branched covering construction. See for example the proof of \cite[Proposition 4.1.6]{Lazerfeld:Positivity}. Consider the map
\begin{equation*}
\pow_{k}: \totalspace(\bigO_{1}|_{Z}) \to \totalspace(\bigO_{k}|_{Z}), \quad \pow_{k} v_{z} = v_{z}^{\otimes k}, \quad z \in p^{-1}(z), z\in Z.
\end{equation*}
This is a $k$-fold branched covering whose branch locus is the zero section. If $T_{k}$ is the tautological section of $p_{k}^{\ast}\bigO_{k}$ over $\totalspace(\bigO_{k}|_{Z})$ where $p_{k}:\totalspace(\bigO_{k}|_{Z})\to Z$ is the projection, then $\pow^{k}Z'_{k}$ is the zero locus of $\multisec_{k} + T_{k}$. This zero locus is exactly the graph of $-\multisec_{k}$ in $\totalspace(\bigO_{k}|_{Z})$. Hence the branch locus of $\pow_{k}|_{Z'_{k}}$ is exactly $\{ \multisec_{k} = 0\} \cap Z = Z'_{k} \cap Z$. Therefore $p_{k}\circ\pow_{k}: Z'_{k} \to Z$ is a holomorphic branched covering with branch locus $Z_{k}' \cap Z = Z_{k} \cap Z$ as claimed.
\end{proof}

\begin{thm}\label{Theorem:Cabling}
Suppose that $(X^{\pm}, \Line^{\pm})$ have divisors with a common projective model. Then for all $k \in \Z_{\geq 1}$, $(X^{\pm}, (\Line^{\pm})^{\otimes k})$ have divisors with a common projective model, yielding pencil pairs $(X^{\pm}, (\Line^{\pm})^{\otimes k}, \phi_{k})$ by Lemma \ref{Lemma:ProjectiveModel}.
\end{thm}

\begin{proof}
This follows immediately from Theorem \ref{Theorem:BranchedCover} and the definitions involved. The projective models for polarized divisors $(Z^{\pm}_{1}, \Line^{\pm})$ are deformation equivalent through complete intersection subvarieties of some $\ModSpace \subset \Proj^{N}$ for which $\Line^{\pm}$ are pull-backs of the line bundle $\bigO_{d_{1}} \to \Proj^{N}$. Theorem \ref{Theorem:BranchedCover} then provides that for all $k$, pairs of divisors and branch-loci for $(X^{\pm}, (\Line^{\pm})^{\otimes k})$ will be symplectomorphic by some $\phi_{k}$. Such $\phi_{k}$ define the desired pencil pairs.
\end{proof}

\section{$\dim_{\C}X^{\pm}=3$ examples}\label{Sec:Dim3}

Now we work out some new examples of positive mapping class relations coming from $\dim_{\C}X^{\pm}=3$ pencil pairs by applying the results of \S \ref{Sec:ProjectiveModel}. First, we briefly review some vocabulary and known results.

The \emph{anticanonical bundle} of a variety $X$ is defined $A_{X} = \wedge_{\C}^{\dim_{\C}X}TX$ for which $c_{1}(A_{X}) = c_{1}(X)$. If $A_{X}$ is ample, $X$ is \emph{Fano} and if $A_{X}$ is holomorphically trivializable, then $X$ is \emph{Calabi-Yau}. If $X$ is Fano,
\be
\item $\ind_{X}$ is the maximal $k$ for which there is an ample line bundle $\Line \to X$ with $A_{X}=\Line^{\otimes k}$, and
\item a generic section of $A_{X}$ cuts out an \emph{anticanonical divisor} $Z$ (which is Calabi-Yau).
\ee
A simply connected $\dim_{\C}=2$ Calabi-Yau variety is a \emph{$K3$ surface}. According to Kodaira, the diffeomorphism type of a $K3$ surface is uniquely determined. See \cite{K3Lectures} for this and additional background. For an ample line bundle $\Line$ over a $K3$ surface $Z$, the \emph{genus} is defined as
\begin{equation*}
g(Z, \Line) = 1 + \half \int_{Z}c_{1}(\Line)^{2},
\end{equation*}
which computes the genus of a divisor in $Z$ cut out by a generic holomorphic section of $\Line$. It follows from a Hodge theory computation that $\chi(Z) = 24$. The adjunction formula applied to anticanonical divisors for a $\dim_{\C}X=3$ Fano variety then gives
\begin{equation}\label{Eq:FanoC1C2}
\int_{X}c_{1}(X)c_{2}(X) = 24
\end{equation}
as well. Therefore the only interesting Chern numbers of a Fano $3$-fold $X$ are
\begin{equation*}
\int_{X} c_{1}(X)^{3} = \int_{X}c_{1}(A_{X})^{3}, \quad \int_{X}c_{2}(X)=\chi(X).
\end{equation*}
Combining these observations with some arithmetic of Chern classes, we obtain the following lemma.

\begin{lemma}\label{Lemma:FanoCritCount}
If $(\multisec_{0}, \multisec_{1})$ is a pencil associated to a line bundle $\Line$ over a variety $X$ of $\dim_{\C}X=3$, then the associated Lefschetz fibration on $X \setminus \{ \multisec_{0} = 0\}$ has
\begin{equation*}
\#\Crit \pi = - \chi(X) + \int_{X}2c_{2}(X)c_{1}(\Line) - 3c_{1}(X)c_{1}(\Line)^{2} + 4c_{1}(\Line)^{3}
\end{equation*}
critical points. In the specific case that $\Line=A_{X}^{\otimes k}$ over a Fano variety $X$,
\begin{equation*}
\#\Crit \pi = - \chi(X) + 48k + k^{2}(4k-3)\int_{X}c_{1}(A_{X})^{3}.
\end{equation*}
\end{lemma}

\begin{proof}
To count $\# \Crit \pi$ we seek to apply the formula from Theorem \ref{Thm:BoundaryRelation}, which requires computations of $\chi(Z_{1})$ and $\chi(B)$ where $Z_{1}$ is a divisor of $X$ cut out by some generic section $\multisec_{1}$ of $\Line \to X$ and $B$ is the base locus of a pencil. Throughout we apply axiomatic properties of Chern classes, cf. \cite{MS:CC}.

Since the $\Line$-valued $1$-form $d\multisec_{1}$ is non-degenerate along $Z_{1}$, the coimage of the map $TX|_{Z_{1}} \to \Line|_{Z_{1}}$, $v_{z} \mapsto d\multisec_{1}(v_{z})$ is a normal bundle to $Z_{1} \subset X$. Thus we have an identification $TX|_{Z_{1}} \simeq TZ_{1} \oplus \Line|_{Z_{1}}$ which when combined with the multiplicativity of total Chern classes yields the adjunction formula
\begin{equation*}
c(TX|_{Z_{1}}) = (1 + c_{1}(\Line))c(Z_{1}).
\end{equation*}
Expansion of the formula at $c_{2}(TX|_{Z_{1}})$ yields the second equality in the following computation:
\begin{equation*}
\begin{aligned}
\chi(Z_{1}) &= \int_{Z_{1}}c_{2}(Z_{1}) = \int_{Z}c_{2}(TX|_{Z_{1}}) - c_{1}(\Line)c_{1}(X) + c_{1}(\Line)^{2}\\
&= \int_{X}c_{1}(\Line)\left( c_{2}(X) - c_{1}(\Line)c_{1}(X) + c_{1}(\Line)^{2} \right).
\end{aligned}
\end{equation*}
The first equality follows from the fact that $c_{2}(Z_{1})$ is the Euler class of $Z_{1}$. The third equality follows from $[Z_{1}] = \PD_{X}c_{1}(\Line)$. Similar arguments provide us with
\begin{equation*}
c(TX|_{B}) = (1 + c_{1}(\Line))^{2}c(TX|_{B}) \implies \chi(B) = \int_{X}c_{1}(\Line)^{2}(c_{1}(X) - 2c_{1}(\Line)).
\end{equation*}

Finally we apply $-\Crit \pi = \chi(X) - 2\chi(Z_{1}) + \chi(B)$ from Theorem \ref{Thm:BoundaryRelation} and add up the computed terms to obtain the general formula. In the Fano case with $\Line = A_{X}^{\otimes k}$, Equation \eqref{Eq:FanoC1C2} yields
\begin{equation*}
\int_{X} 2c_{2}(X)c_{1}(\Line) = \int_{X} 2c_{2}(X)c_{1}(A_{X}^{\otimes k}) = \int_{X} 2kc_{2}(X)c_{1}(A_{X}) =  \int_{X} 2kc_{2}(X)c_{1}(X) = 48k.
\end{equation*}
Substituting $c_{1}(\Line) = kc_{1}(A_{X})$ into the remaining terms of the general formula then gives the second formula in the statement of the lemma.
\end{proof}

\subsection{Classification of polarized $K3$ surfaces of low genus}\label{Sec:K3Classification}

Let $\Line \to X$ be very ample line bundle over a $\dim_{\C}X=3$ Fano variety  with $\Line^{\otimes \ind_{X}}=A_{X}$. This implies that $\Line|_{Z}$ is very ample as well for $Z \subset X$ an anticanonical divisor. The following theorem reviews classification results of such polarized $K3$ surfaces $(Z, \Line)$ by Saint-Donat \cite{SD:ProjectiveModel} (for $g(Z, \Line) \leq 5$) and Mukai \cite{Mukai:Gleg10} (for $6 \leq g(Z, \Line) \leq 10$).

\begin{thm}\label{Thm:K3Classification}
In the above notation, suppose that $g=g(Z, \Line) \leq 10$ where it may be computed
\begin{equation*}
g = 1+ \half \int_{Z}c_{1}(\Line)^{2}= 1 + \half \ind_{X}\int_{X} c_{1}(\Line)^{3} = 1 + \half \ind_{X}^{-2}\int_{X} c_{1}(A_{X})^{3}.
\end{equation*}
Then $(Z, \Line)$ has a projective model $(\ModSpace_{g} \subset \Proj^{N_{g}}, \vec{d})$ which is entirely determined by $g$.
\end{thm}

For $g=3$, $Z$ is a $\deg=4$ hypersurface in $\Proj^{3}$ with $\Line = \bigO_{1}$ so the model is $(\ModSpace_{3} = \Proj^{4} \subset \Proj^{4}, \vec{d}=(1,4))$. For $g=4$, $Z$ is an intersection of $\deg=2$ and $\deg=3$ hypersurfaces in $\Proj^{4}$ with $\Line = \bigO_{1} \rightarrow \Proj^{4}$, so the model is $(\ModSpace_{5} = \Proj^{5}, \vec{d}=(1,2,3))$. For $g = 5$, $Z$ is an intersection of three $\deg=2$ hypersurfaces in $\Proj^{5}$ with $\Line = \bigO_{1} \rightarrow \Proj^{5}$, so $(\ModSpace_{5} = \Proj^{5}, \vec{d}=(1,2,2,2))$. In the $6 \leq g \leq 10$ cases, the $\ModSpace_{g}$ have positive codimension in $\Proj^{N_{g}}$ and will be described in the following subsection.

\subsection{Examples of pencil pairs from anticanonical divisors on Fano $3$-folds}

Combining the above results with Theorems \ref{Lemma:ProjectiveModel} and \ref{Theorem:Cabling}, we have the following theorem.

\begin{thm}\label{Thm:K3PencilPair}
Let $(X^{\pm}, \Line^{\pm})$ be polarized Fano $3$-folds for which $A_{X^{\pm}} = (\Line^{\pm})^{\ind_{X^{\pm}}}$ and $\Line^{\pm}$ are very ample. Suppose that their anticanonical divisors $Z^{\pm} \subset X^{\pm}$ are such that $g(Z^{\pm}, \Line^{\pm}) \leq 10$, or equivalently
\begin{equation*}
\int_{X^{\pm}} A_{X^{\pm}}^{3} \leq 18\ind_{X^{\pm}}^{2}
\end{equation*}
Then for every $k \in \Z_{>0}$ there are $\phi_{k}$ for which $(X^{\pm}, A_{X^{\pm}}^{\otimes k}, \phi_{k})$ are pencil pairs.
\end{thm}

Indeed, under the hypothesis of Theorem \ref{Thm:K3PencilPair}, Theorem \ref{Thm:K3Classification} provides a projective model for an anticanonical divisor of each $X^{\pm}$ which are $K3$ surfaces. These projective models are entirely determined the genus, so the $(X^{\pm}, A_{X^{\pm}})$ have divisors with a common projective model (Definition \ref{Def:CommonProjectiveModel}). By Lemma \ref{Lemma:ProjectiveModel}, there is a symplectomorphism $\phi=\phi_{1}$ for which the $(X^{\pm}, A_{X^{\pm}}, \phi)$ form a pencil pair. By Theorem \ref{Theorem:Cabling}, there are $\phi_{k}$ for which the $(X^{\pm}, A_{X^{\pm}}^{\otimes k}, \phi_{k})$ are pencil pairs for all $k \in \Z_{>0}$.

To find all examples of pencil pairs determinable from Theorem \ref{Thm:K3PencilPair}, we cross reference the above theorems with the complete classification of $\dim_{\C}=3$ Fano varieties by Iskovskih \cite{Iskovskih:I, Iskovskih:II} and Mori-Mukai \cite{MoriMukai}. There are $105$ such families of varieties which are listed in the Fanography database \cite{Fanography} along with relevant data such as indices and Hodge diamonds (which of course determine Euler characteristics).\footnote{See \url{https://www.fanography.info/about} for a complete list of resources from which the database is compiled.}

We collect the relevant examples of our Fano $X$ into families (having at least two members) organized by their $\ind_{X}$ and $\int_{X}c_{1}(A_{X})^{3}$. Members of the families are labeled by their indices $a$-$b$ as in \cite{Fanography}, with $a$ being the Picard rank. For each, a short description of $X$ is given. According to Theorem \ref{Thm:K3PencilPair} any two members of a given family form a pencil pair.

\subsubsection{Case $\ind_{X}=1, \int_{X}c_{1}(A_{X})^{3}=10 \implies g(Z, A_{X})=6$}
\be
\item[(1-5)] $X$ is an intersection of $\ModSpace_{6}=G(2, 5) \subset \Proj^{9}$ with two hyperplanes and a quadric for which $\chi(X)=-16$.
\item[(2-4)] $X$ is a divisor of $\deg=(1, 3)$ on $\ProjOne \times \Proj^{3}$ -- that is, a blow up of $\Proj^{3}$ along an intersection of two cubic hypersurfaces -- having $\chi(X)=-14$.
\ee
Here is the projective model is $(\ModSpace_{6} \subset \Proj^{9}, \vec{d}=(1,1,1,2))$ following \cite{Mukai:Gleg10}. We obtain Weinstein $4$-manifolds $(F_{k}, \beta_{k})$ given by a generic $A_{X}^{\otimes k}$ divisor in either $X$, minus its intersection with another such generic divisor (that is the base locus of an associated pencil). The $F_{k}$ are smooth fibers of Lefschetz fibrations on $X$ minus a divisor for $A_{X}^{\otimes k}$ associated to the pencil for $A_{X}^{\otimes k}$. Since these divisors can be viewed as hyperplane sections in a projective variety, we can also view the $F_{k}$ as affine varieties.

Applying Theorem \ref{Thm:K3PencilPair}, we have factorizations of $\tau_{\partial F_{k}}$ into products of
\begin{equation*}
\chi + 48k + 10k^{2}(4k-3), \quad \chi=-16, -14
\end{equation*}
Dehn twists along Lagrangian spheres in $F_{k}$.

\subsubsection{Case $\ind_{X}=1, \int_{X}c_{1}(A_{X})^{3}=12 \implies g(Z, A_{X})=7$}
\be
\item[(1-6)] $X$ is an intersection of $\ModSpace_{7} \subset \Proj^{15}$ of $\dim_{\C}\ModSpace=10$ with four hyperplanes, with $\chi(X) = -10$ as described in \cite{Mukai:Gleg10}.
\item[(2-5)] $X$ is an intersection of $\deg=(0,3)$ and $\deg=(1,1)$ hypersurfaces in $\Proj^{1} \times \Proj^{4}$ (that is, a blow up of a cubic in $\Proj^{4}$ along an intersection of two hyperplane sections) having $\chi(X)=-6$.
\item[(2-6)] $X$ is a $\deg=(2,2)$ divisor in $\Proj^{2} \times \Proj^{2}$ having $\chi(X) = -12$.
\item[(3-1)] $X$ is a double branched cover of $(\ProjOne)^{3}$ over a $\deg=(2,2,2)$ divisor having $\chi(X)=-8$.
\ee
In this case the projective model of \cite{Mukai:Gleg10} is $(\ModSpace_{7} \subset \Proj^{15}, \vec{d}=(1,1,1,1,1))$.

Theorem \ref{Thm:K3PencilPair} gives us Weinstein $4$-manifolds $(F_{k}, \beta_{k})$ and factorizations of $\tau_{\partial F_{k}}$ into products of
\begin{equation*}
- \chi + 48k + 12k^{2}(4k-3), \quad \chi=-10,-6,-12,-8
\end{equation*}
Dehn twists along Lagrangian spheres in $F_{k}$.

\subsubsection{Case $\ind_{X}=1, \int_{X}c_{1}(A_{X})^{3}=14 \implies g(Z, A_{X})=8$}
\be
\item[(1-7)] $X$ is a the intersection of $\ModSpace_{8}=G(2, 6) \subset  \Proj^{14}$ with nine hyperplanes with $\chi(X)=-6$.
\item[(2-7)] $X$ is an intersection of $\deg=(0,2)$ and $\deg=(1,2)$ divisors on $\ProjOne \times \Proj^{4}$ with $\chi(X)=-4$.
\item[(2-8)] Let $Y$ be the $\Proj^{1}$-bundle given by projectivizing the $\C^{2}$-bundle $\C \oplus \bigO_{1} \rightarrow \Proj^{2}$. Then $X$ is a double branched cover of $Y$ whose branch locus is an anticanonical divisor in $Y$ with $\chi(X)=-12$.
\item[(3-2)] $X$ is a divisor in a $\Proj^{2}$ bundle over $\Proj^{1}\times \Proj^{1}$ having $\chi(X)=2$.
\ee

Theorem \ref{Thm:K3PencilPair} gives us Weinstein $4$-manifolds $(F_{k}, \beta_{k})$ and factorizations of $\tau_{\partial F_{k}}$ into products of
\begin{equation*}
- \chi + 48k + 14k^{2}(4k-3), \quad \chi=-6,-4,-12,2
\end{equation*}
Dehn twists along exact Lagrangian spheres.

\subsubsection{Case $\ind_{X}=1, \int_{X}c_{1}(A_{X})^{3}=16 \implies g(Z, A_{X})=9$}
\be
\item[(1-8)] $X$ is an intersection of a $\dim_{\C}=6$ spin Grassmann variety in $\Proj^{13}$ with three hyperplanes, having $\chi(X)=-2$.
\item[(2-9)] $X$ is an intersection of $\deg=(1,1)$ and $\deg=(1,2)$ divisors in $\Proj^{2} \times \Proj^{3}$ having $\chi(X)=-4$.
\item[(2-10)] $X$ is an intersection of $\deg=(0,2)$, $\deg=(0,2)$, and $\deg=(1,1)$ divisors in $\ProjOne \times \Proj^{5}$ with $\chi(X)=0$.
\ee
Again the first example corresponds to the $g=9$ model of \cite{Mukai:Gleg10}.

Theorem \ref{Thm:K3PencilPair} gives us Weinstein $4$-manifolds $(F_{k}, \beta_{k})$ and factorizations of $\tau_{\partial F_{k}}$ into products of
\begin{equation*}
- \chi + 48k + 16k^{2}(4k-3), \quad \chi=-2,-4,0
\end{equation*}
Dehn twists along Lagrangian spheres.

\subsubsection{Case $\ind_{X}=1, \int_{X}c_{1}(A_{X})^{3}=18 \implies g(Z, A_{X})=10$}
\be
\item[(1-9)] $X$ is a $\dim_{\C}=5$ Grassmann bundle for the exceptional Lie group $G_{2}$ inside of $\Proj^{13}$, intersected with two hyperplanes having $\chi(X)=0$.
\item[(2-11)] $X$ is a blowup of a $\deg=3$ hypersurface in $\Proj^{4}$ along a line having $\chi(X)=-4$.
\item[(3-3)] $X$ is a $\deg=(1,1,2)$ divisor on $\Proj^{1} \times \Proj^{1} \times \Proj^{2}$ having $\chi(X)=2$.
\item[(3-4)] $X$ is obtained by taking a branched cover of $\ProjOne \times \Proj^{2}$ along a $\deg=(2,2)$ divisor and blowing up along a smooth fiber of the projection to $\Proj^{2}$, having $\chi(X)=4$.
\item[(8-1)] $X$ is a blowup of $\Proj^{2}$ at six points times $\ProjOne$ having $\chi(X)=18$.
\ee
Again, the first member of the family corresponds to the model of \cite{Mukai:Gleg10}.

Theorem \ref{Thm:K3PencilPair} gives us Weinstein $4$-manifolds $(F_{k}, \beta_{k})$ and factorizations of $\tau_{\partial F_{k}}$ into products of
\begin{equation*}
- \chi + 48k + 12k^{2}(4k-3), \quad \chi=-4,0,4,2,18
\end{equation*}
Dehn twists along exact Lagrangian spheres. The case $k=1$ is Theorem \ref{Thm:MainExample}.

\subsubsection{Case $\ind_{X}=2, \int_{X}c_{1}(A_{X})^{3}=48 \implies g(Z, \Line)=7$}\label{Sec:ObaDouble}
\be
\item[(2-32)] $X$ is a divisor of $\deg=(1,1)$ in $\Proj^{2} \times \Proj^{2}$ with $\chi(X)=6$. Here $\Line$ is of degree $(1,1)$ so that $\Line^{\otimes 2} = A_{X}$. Since $\Line$ is the pullback of $\bigO_{1}$ from the Segre embedding $\Proj^{2} \times \Proj^{2} \rightarrow \Proj^{8}$ and so is very ample.
\item[(3-27)] $X= (\Proj^{1})^{3}$ having $\chi(X)=8$. Here $\Line$ has $\deg=(1,1,1)$ giving $\Line^{\otimes 2} = A_{X}$. Again $\Line$ is very ample as it is determined by a Segre embedding.
\ee

As mentioned in \S \ref{Sec:K3Classification}, the $K3$ surfaces determined as anticanonical divisors have projective models described in \cite{Mukai:Gleg10}. These $X$ are exactly the Fano $3$-folds considered in \cite{Oba:FourDRelation}. Here Theorem \ref{Thm:K3PencilPair} gives us Weinstein $4$-manifolds $(F_{k}, \beta_{k})$ given as the smooth fibers of Lefschetz fibrations associated to the amply line bundles $A_{X}^{\otimes k} \to X$ and factorizations of $\tau_{\partial F_{k}}$ into products of
\begin{equation}\label{Eq:K3ObaCount}
- \chi + 48k(4k^{2} - 3k + 1), \quad \chi=6,8
\end{equation}
Dehn twists along exact Lagrangian spheres. We will generalize these examples in the next section by considering Lefschetz pencils associated to $\Line^{\otimes k}$ (rather than $A_{X}^{\otimes k} = \Line^{\otimes 2k}$) as considered in \cite{Oba:FourDRelation}.

The complement of a divisor for $\Line$ in the $X \subset \Proj^{2} \times \Proj^{2}$ indexed (2-32) can be identified with a cotangent disk bundle $\disk^{\ast}\Proj^{2}$, and is studied from a symplectic point of view in \cite{Audin:LagrangianSkeletons, Biran:LagrangianBarriers, Torricelli:Projective}. In particular, \cite[\S 5]{Torricelli:Projective} describes a Lefschetz fibration associated to $\Line$ on this Weinstein $8$-manifold.

\subsection{Oba's example revisited}\label{Sec:ObaRevisited}

As in \S \ref{Sec:ObaDouble} above, consider $(X^{+}, \Line^{+})$ given by a $\deg=(1,1)$ hypersurface in $\Proj^{2} \times \Proj^{2}$ equipped with $\Line^{+}$ a $\deg=(1,1)$ line bundle and $(X^{-}, \Line^{-})$ is $(\ProjOne)^{3}$ with $\Line^{-}$ a $\deg=(1,1,1)$ line bundle. Then divisors $Z^{\pm}_{1} \subset X^{\pm}$ for the line bundles $\Line^{\pm} \to X^{\pm}$ are del Pezzo surfaces (that is, Fano $2$-folds) of degree $\int_{Z_{1}^{\pm}} A_{Z_{1}^{\pm}}^{2}=6$. In these cases the $\Line^{\pm}|_{Z_{1}^{\pm}}$ are the anticanonical bundles $A_{Z^{\pm}_{1}}$. Smoothly, we have equivalences
\begin{equation*}
Z^{\pm}_{1} = \Proj^{2}\#3\overline{\Proj^{2}} = \left(\Proj^{1} \times \Proj^{1}\right)\#2\overline{\Proj^{2}}.
\end{equation*}

Polarized del Pezzo surfaces $(Z, \Line=A_{X})$ of $\int_{Z} A_{Z}^{2}=6$ are isomorphic by \cite{Fujita}. Any such biholomorphism $\phi: Z^{+}_{1} \rightarrow Z^{-}_{1}$ lifts to an isomophism of anticanonical bundles, by the fact that the anticanonical bundles are constructed from the $TZ^{\pm}_{1}$. We can then consider either $(Z^{\pm}_{1}, \Line^{\pm})$ as being a projective model and obtain a pencil pair $(X^{\pm}, \Line^{\pm}, \phi)$ from Lemma \ref{Lemma:ProjectiveModel}. Applying $l$-fold cables as in Theorem \ref{Theorem:Cabling}, we get pencil pairs $(X^{\pm}, (\Line^{\pm})^{\otimes l}, \phi_{l})$ for all $l \geq 1$ which recover Oba's example \cite{Oba:FourDRelation} in the case $l=1$, and recover the examples of \S \ref{Sec:ObaDouble} when $k=2l$. The $l=2$ case gives us the polarized $K3$ surfaces $(Z^{\pm}_{1}, (\Line^{\pm})^{2}=A_{X})$ of $g=7$.

For the associated mapping class relations, we have $(F_{l}, \beta_{l})$ given by a $\deg=(l,l,l)$ hypersurface in $X^{-}=(\ProjOne)^{3}$ minus its intersection with another hypersurface of the same degree. Writing $\omega_{j} \in H^{2}(X^{-})$ for the class corresponding to a volume form on the $j$th factor of the the product, we have
\begin{equation*}
c(X^{-}) = (1 + 2\omega_{1})(1 + 2\omega_{2})(1 + 2\omega_{3}), \quad c_{1}((\Line^{-})^{\otimes l}) = l(\omega_{1} + \omega_{2} + \omega_{3}).
\end{equation*}
The number of critical points of a Lefschetz fibration $\pi^{-}_{l}$ associated to $(X^{-}, (\Line^{-})^{\otimes l})$ may then be computed using the first formula in Lemma \ref{Lemma:FanoCritCount}. From Theorem \ref{Thm:BoundaryRelation} we can count the number of critical points of a Lefschetz fibration $\pi^{+}_{l}$ associated $(X^{+}, (\Line^{+})^{\otimes l})$ using $\# \Crit \pi^{-}_{l}$ and the $\chi(X^{\pm})$. The counts are
\begin{equation*}
\#\Crit \pi^{\pm}_{l} = -\chi(X^{\pm}) + 12l(2l^{2} - 3l + 2) , \quad \chi(X^{+}) = 6, \quad \chi(X^{-}) = 8.
\end{equation*}
As expected, the cases $2k=l$ recover the counts of Equation \eqref{Eq:K3ObaCount}.

\textsc{CNRS, Institut de Mathématiques de Jussieu-Paris Rive Gauche, Sorbonne Université, Paris, France}\par\nopagebreak
\textit{Email:} \href{mailto:russell@imj-prg.fr}{russell@imj-prg.fr}\par\nopagebreak
\textit{URL:} \href{https://www.russellavdek.com/}{russellavdek.com}

\end{document}